\documentclass[11pt]{article}
\usepackage{fullpage}

\usepackage{amssymb,amsmath,amsthm}
\usepackage{url}
\usepackage{verbatim}

\newtheorem{thm}{Theorem}

\newtheorem{lem}[thm]{Lemma}

\newtheorem{problem}[thm]{Problem}

\theoremstyle{definition}

\newtheorem{remark}[thm]{Remark}

\usepackage{indentfirst}
\usepackage{xspace}
\usepackage[colorlinks=true,linkcolor=blue,citecolor=blue,urlcolor=blue]{hyperref}

\def\G{\mbox{\ensuremath{\mathcal G}}\xspace}
\def\HH{\mbox{\ensuremath{\mathcal H}}\xspace}
\def\K{\mbox{\ensuremath{\mathcal K}}\xspace}

\begin{document}
\title{Note on polychromatic coloring of hereditary hypergraph families II}
\author{D\"om\"ot\"or P\'alv\"olgyi
\thanks{ELTE E\"otv\"os Lor\'and University and Alfr\'ed R\'enyi Institute of Mathematics, Budapest, Hungary. Email: \texttt{domotor.palvolgyi@ttk.elte.hu}. Supported by the NRDI EXCELLENCE-24 grant no.~151504 Combinatorics and Geometry and by the ERC Advanced Grant no.~101054936 ERMiD.}}
\maketitle

\noindent\textbf{Disclaimer.} This manuscript, including this disclaimer, was written by ChatGPT 5.5 Thinking, based on an interactive discussion with D\"om\"ot\"or P\'alv\"olgyi. The mathematical correctness of the manuscript has been verified by the author.

\begin{abstract}
We extend a recent construction concerning polychromatic colorings of hereditary hypergraph families.
For every integer $h\ge 4$ we construct a $(2h-1)$-uniform hypergraph which has no polychromatic $3$-coloring, but all of whose $h$-heavy restricted subhypergraphs are $2$-colorable.
Together with the previously known case $h=3$, this gives examples with uniformity $2h-1$ for every $h\ge 3$.
The construction is based on complements of suitable $h$-uniform hypergraphs on $3h-1$ vertices.
For $h\ge 9$ we prove existence by a simple probabilistic argument; the remaining cases $4\le h\le 8$ are certified by a short exhaustive computer check, whose fully reproducible description and source code are included in the appendix.
\end{abstract}


\textbf{Mathematics Subject Classification (MSC):} 05C15 (Coloring of graphs and hypergraphs)

\medskip

A hypergraph $\HH=(V,E)$ is a collection of subsets of a vertex set $V$.
It is $m$-uniform if all its edges have size $m$, and it is $h$-heavy if all its edges have size at least $h$.
For $X\subset V$, the trace of $\HH$ on $X$ is
\[
        \HH[X]=(X,\{e\cap X:e\in E(\HH)\}).
\]
A subhypergraph of a trace is called a restricted subhypergraph.
A coloring of the vertices of $\HH$ with colors $1,\ldots,k$ is polychromatic if every edge of $\HH$ contains all $k$ colors.
For $k=2$, this is the usual proper $2$-coloring of a hypergraph, also known as Property B.

The starting point is the following result from \cite{Pal-note}. For background on geometric hypergraph colorings and their relation to cover-decomposition, see the surveys \cite{survey,new-survey}.

\begin{thm}[\cite{Pal-note}]\label{thm:h3}
There exists a $5$-uniform hypergraph which has no polychromatic $3$-coloring, but all of whose $3$-heavy restricted subhypergraphs are $2$-colorable.
\end{thm}

Here we prove the corresponding statement for every larger value of the heaviness parameter.

\begin{thm}\label{thm:main}
For every integer $h\ge 4$, there exists a $(2h-1)$-uniform hypergraph $\HH$ which has no polychromatic $3$-coloring, but all of whose $h$-heavy restricted subhypergraphs are $2$-colorable.
\end{thm}

\noindent\textbf{Roadmap.}
Section~\ref{sec:complement} gives the complement construction and reduces Theorem \ref{thm:main} to three elementary conditions on a family of $h$-sets.
Section~\ref{sec:random} proves the existence of such families for all $h\ge 9$ by a random construction and finishes the proof of the theorem modulo the finite cases.
Section~\ref{sec:barrier} examines whether a natural covering-design method might reach uniformity $2h+c$ for fixed $c\ge 0$, and explains why the calculation becomes unfavorable as soon as $c\ge 1$, while the case $c=0$ remains a delicate possible next step.
Appendix~\ref{app:finite} gives the explicit deterministic data for $4\le h\le 8$, Appendix~\ref{app:code} contains the verification code, and Appendix~\ref{app:process} briefly records how this manuscript arose from the prompting process.

Combining Theorems \ref{thm:h3} and \ref{thm:main}, we get examples of uniformity $2h-1$ for every $h\ge 3$.

\section{A complement construction}\label{sec:complement}

The construction is most naturally described through complements.
Fix an integer $h\ge 2$ and put
\[
        n=3h-1.
\]
Let $V$ be an $n$-element set.
We shall choose a family $\K\subset {V\choose h}$.
The sets in $\K$ will be the missing $h$-sets of an auxiliary hypergraph.  Define
\[
        \G={V\choose h}\setminus \K,
\]
and then define a $(2h-1)$-uniform hypergraph $\HH$ by
\[
        E(\HH)=\{V\setminus B:B\in \G\}.
\]
Thus the edges of $\HH$ are the complements of the edges of $\G$.

We need three properties of $\K$.

\begin{description}
\item[(K1)] For every $A\in {V\choose h-1}$, not all $h$-sets of the form $A\cup\{v\}$, $v\in V\setminus A$, belong to $\K$.  Equivalently, among the $2h$ extensions of $A$ to an $h$-set, at least one is not in $\K$.

\item[(K2)] For every $X\in {V\choose 2h}$, there is an $h$-set $R\subset X$ such that
\[
        R\in \K
        \quad\hbox{and}\quad
        X\setminus R\in \K .
\]

\item[(K3)] For every $Y\in {V\choose h}$, there is a set $U\subset V\setminus Y$, $|U|=h-1$, such that
\[
        U\cup\{y\}\in \K
        \qquad\hbox{for every }y\in Y.
\]
\end{description}

\begin{lem}\label{lem:sufficient}
If $\K\subset {V\choose h}$ satisfies {\rm (K1)}, {\rm (K2)}, and {\rm (K3)}, then the hypergraph $\HH$ defined above has no polychromatic $3$-coloring, but all of its $h$-heavy restricted subhypergraphs are $2$-colorable.
\end{lem}

\begin{proof}
First we show that $\HH$ has no polychromatic $3$-coloring.
Consider any $3$-coloring of $V$.
One color class $S$ has size at most
\[
        \left\lfloor {3h-1\over 3}\right\rfloor=h-1.
\]
Extend $S$, if necessary, to an $(h-1)$-set $A$.
By (K1), some set $A\cup\{v\}$ with $v\in V\setminus A$ does not belong to $\K$.
Hence
\[
        B=A\cup\{v\}\in \G.
\]
The corresponding edge $V\setminus B$ of $\HH$ misses the color class $S$.
Thus the coloring is not polychromatic.

It remains to prove that every $h$-heavy trace of $\HH$ is $2$-colorable.
Since every subhypergraph of a $2$-colorable hypergraph is $2$-colorable, it is enough to color the $h$-heavy part of every trace $\HH[X]$.

Let $X\subset V$.
We give a red-blue coloring of $X$ with no monochromatic trace edge of size at least $h$.

If $|X|\le 2h-2$, color $X$ as evenly as possible.
Then both color classes have size at most $h-1$, so no monochromatic $h$-heavy trace can occur.

Suppose next that $|X|=2h-1$.
Let $Y=V\setminus X$, so $|Y|=h$.
By (K3), there is a set $U\subset X$, $|U|=h-1$, such that
\[
        U\cup\{y\}\in \K
        \qquad\hbox{for every }y\in Y.
\]
Color $U$ blue and color $X\setminus U$ red.
The blue class has size $h-1$, so it contains no monochromatic $h$-heavy trace.
If there were a red monochromatic trace of size at least $h$, then, since the red class has size exactly $h$, there would be some $B\in \G$ such that
\[
        X\setminus B=X\setminus U.
\]
Equivalently,
\[
        B\cap X=U.
\]
As $B$ has size $h$ and $U$ has size $h-1$, this means $B=U\cup\{y\}$ for some $y\in Y$, contradicting the choice of $U$, because all such sets belong to $\K$.

Now suppose that $|X|=2h$.
By (K2), write $X=R\sqcup S$ with
\[
        |R|=|S|=h,
        \qquad
        R,S\in \K .
\]
Color $R$ red and $S$ blue.
If there were a red monochromatic trace of size at least $h$, then it would have to be all of $R$, and hence for some $B\in \G$ we would have
\[
        X\setminus B=R.
\]
Thus $B\cap X=S$.
Since $|B|=|S|=h$, this forces $B=S$, contradicting $S\in\K$ and $B\in\G$.
The blue case is symmetric.

Finally suppose that $|X|\ge 2h+1$.
Choose any $2h$-subset $X_0\subset X$.
By (K2), there is an $h$-set $R\subset X_0$ with $R\in \K$.
Color $R$ red and color $X\setminus R$ blue.

A red monochromatic trace of size at least $h$ would have to be exactly $R$.
This would imply $X\setminus R\subset B$ for some $B\in\G$, impossible because
\[
        |X\setminus R|\ge h+1
        \quad\hbox{while}\quad
        |B|=h.
\]
A blue monochromatic trace would imply $R\subset B$ for some $B\in\G$.
Since $|R|=|B|=h$, this would force $B=R$, contradicting $R\in\K$.
This completes the proof.
\end{proof}

\section{The random construction}\label{sec:random}

We now prove that suitable families $\K$ exist for all $h\ge 9$.

Choose every $h$-subset of $V$, independently, to belong to $\K$ with probability $q$.
We estimate the probability that one of (K1), (K2), (K3) fails.

For a fixed $(h-1)$-set $A$, the probability that all $2h$ extensions $A\cup\{v\}$ belong to $\K$ is $q^{2h}$.
Thus the expected number of violations of (K1) is
\[
        {3h-1\choose h-1}q^{2h}.
\]

For a fixed $2h$-set $X$, the complementary pairs
\[
        R,\ X\setminus R
        \qquad (R\in {X\choose h})
\]
partition the family ${X\choose h}$ into $\frac12{2h\choose h}$ pairs.
The events that a given complementary pair has both members in $\K$ are therefore mutually independent, since every $h$-set of $X$ appears in exactly one pair.
Hence the probability that no complementary pair has both members in $\K$ is
\[
        (1-q^2)^{\frac12{2h\choose h}}.
\]
Thus the expected number of violations of (K2) is at most
\[
        {3h-1\choose 2h}(1-q^2)^{\frac12{2h\choose h}}.
\]

Finally fix $Y\in {V\choose h}$.
For a fixed $U\subset V\setminus Y$, $|U|=h-1$, the event
\[
        U\cup\{y\}\in\K
        \qquad\hbox{for all }y\in Y
\]
has probability $q^h$.
For different choices of $U$, these events involve disjoint $h$-sets: an $h$-set with exactly one point in $Y$ determines both its unique point of $Y$ and its $(h-1)$-set outside $Y$.
Hence the events are independent.
There are ${2h-1\choose h-1}$ choices for $U$.
Therefore the probability that $Y$ violates (K3) is
\[
        (1-q^h)^{{2h-1\choose h-1}}.
\]
The expected number of violations of (K3) is
\[
        {3h-1\choose h}(1-q^h)^{{2h-1\choose h-1}}.
\]

Consequently, if
\[
\begin{aligned}
\Phi_h(q)=&
{3h-1\choose h-1}q^{2h}
+
{3h-1\choose 2h}(1-q^2)^{\frac12{2h\choose h}}\\
&+
{3h-1\choose h}(1-q^h)^{{2h-1\choose h-1}}
<1,
\end{aligned}
\tag{1}\label{eq:Phi}
\]
then a suitable family $\K$ exists.

\begin{lem}\label{lem:large}
For every $h\ge 9$, there is a family $\K\subset {V\choose h}$ satisfying {\rm (K1)}, {\rm (K2)}, and {\rm (K3)}.
\end{lem}

\begin{proof}
For $9\le h\le 30$, the following table gives a value of $q$ for which $\Phi_h(q)<1$.

\[
\begin{array}{c|c|c}
h&q&\Phi_h(q)\\ \hline
9&0.447&0.886370744736\\
10&0.430&0.510452111104\\
11&0.415&0.285146456892\\
12&0.403&0.154487361262\\
13&0.393&0.081897387503\\
14&0.384&0.042685390818\\
15&0.376&0.021870669774\\
16&0.369&0.011061930960\\
17&0.363&0.005547730636\\
18&0.357&0.002755290720\\
19&0.350&0.009734044111\\
20&0.350&0.000809153106\\
21&0.350&0.000653057474\\
22&0.350&0.000527692530\\
23&0.350&0.000426826780\\
24&0.350&0.000345562253\\
25&0.350&0.000280009093\\
26&0.350&0.000227070270\\
27&0.350&0.000184274427\\
28&0.350&0.000149645493\\
29&0.350&0.000121600542\\
30&0.350&0.000098869485
\end{array}
\]

For completeness, let us also record a simple analytic estimate for all $h\ge 31$.
Take $q=7/20$.
Using
\[
        {3h\choose h}\le \left({27\over 4}\right)^h
        \quad\hbox{and}\quad
        {2h\choose h}\ge {4^h\over 2h+1},
\]
the three terms of \eqref{eq:Phi} are bounded respectively by
\[
        \left({1323\over 1600}\right)^h,
\]
\[
        \left({27\over 4}\right)^h
        \exp\left(-{49\over 400}\cdot {4^h\over 2(2h+1)}\right),
\]
and
\[
        \left({27\over 4}\right)^h
        \exp\left(-{(7/5)^h\over 2(2h+1)}\right).
\]
At $h=31$ these upper bounds are already less than
\[
        0.003,\qquad 10^{-10^6},\qquad 10^{-90},
\]
respectively. The first bound decreases because $1323/1600<1$.
For the second and third bounds, the logarithms of the ratios of two consecutive bounds are respectively
\[
\log {27\over4}-{49\over400}\left({4^{h+1}\over 2(2h+3)}-{4^h\over 2(2h+1)}\right)
\]
and
\[
\log {27\over4}-\left({(7/5)^{h+1}\over 2(2h+3)}-{(7/5)^h\over 2(2h+1)}\right),
\]
which are negative for every $h\ge 31$.
Thus $\Phi_h(7/20)<1$ for every $h\ge 31$.
The union bound proves the lemma.
\end{proof}

\begin{proof}[Proof of Theorem \ref{thm:main}]
For $4\le h\le 8$, take the family $\K_h$ from Lemma \ref{lem:small} (from Appendix~\ref{app:finite}).
For $h\ge 9$, take a family $\K$ whose existence is guaranteed by Lemma \ref{lem:large}.
In both cases, construct $\HH$ from $\K$ as in Section~\ref{sec:complement}.
Lemma \ref{lem:sufficient} shows that $\HH$ is $(2h-1)$-uniform, has no polychromatic $3$-coloring, and has all $h$-heavy restricted subhypergraphs $2$-colorable.
\end{proof}

\begin{remark}
The construction is not meant to optimize the number of edges.
For $h\ge 9$, the random family $\K$ has density about $q$, and hence the auxiliary family $\G={V\choose h}\setminus\K$ is quite dense.
The point is that the complements of the edges of $\G$ have exactly the uniformity $2h-1$, matching the previously known example for $h=3$.
\end{remark}

\section{A barrier for this method beyond \texorpdfstring{$2h$}{2h}}\label{sec:barrier}

Let us examine whether the same complement-covering method could plausibly give uniformity
\[
        m=2h+c,
\]
where \(c\ge 0\) is fixed.  The first plausible value of \(n\) in this framework is
\[
        n=3h+3c+2.
\]
Indeed, for smaller values of \(n\) the covering condition itself forces a monochromatic \(h\)-heavy trace on a suitable set \(X\) of size about \(2h+c\).  With this choice of \(n\), the auxiliary blocks have size
\[
        n-m=h+2c+2
\]
and must cover every \((h+c)\)-set.
The case \(c=0\) is exactly the first possible improvement over Theorem~\ref{thm:main}: it asks for a \(2h\)-uniform construction on \(3h+2\) vertices whose auxiliary \(h+2\)-sets cover all \(h\)-sets.

A near-optimal random covering has local density roughly
\[
        {1\over {2h+2c+2\choose c+2}},
\]
the reciprocal of the number of auxiliary blocks containing a fixed \((h+c)\)-set.  Now fix a \((2h-1)\)-set \(X\) and an \((h-1)\)-set \(U\subset X\).  The number of auxiliary blocks \(B\) with \(B\cap X=U\) is
\[
        {h+3c+3\choose 2c+3}.
\]
Therefore the expected number of blocks realizing this particular \(U\) is
\[
        { {h+3c+3\choose 2c+3}\over {2h+2c+2\choose c+2}}
        =\Theta_c(h^{c+1}).
\]
On the other hand, the number of possible choices of \(U\) is
\[
        {2h-1\choose h-1}=\exp(\Theta(h)).
\]
Thus for \(c\ge 1\) the expected number of realizations of each fixed \(U\) is already superlinear in \(h\), and the heuristic probability that \(U\) remains a hole is \(\exp(-\Theta_c(h^{c+1}))\).  This is much smaller than the inverse of the number of possible \(U\)'s.  Consequently, even for one fixed \(X\), a locally random near-Steiner covering should leave no holes at all with high probability.  Then the trace of \(\HH\) on \(X\) contains the complete \(h\)-uniform hypergraph \(K_{2h-1}^{(h)}\), and hence is not \(2\)-colorable.

This is not a non-existence result for \((2h+c)\)-uniform examples; a more structured construction might avoid the random-covering behaviour described above.  It does show, however, that the covering-design method has a natural barrier after \(m=2h\).  For every fixed \(c\ge 1\), the expected number of holes in the decisive \((2h-1)\)-traces becomes negligible.  For \(c=0\), the same calculation gives only \(\Theta(h)\) realizations of a fixed \(U\), so it does not rule out the possibility of many holes.  This suggests that the \(2h\)-uniform case might still be accessible, but probably only through a carefully constructed near-Steiner covering, such as one produced by a random-greedy or nibble-type argument rather than by independent random selection.

\begin{problem}\label{prob:2h}
Is the uniformity \(2h-1\) best possible?
Equivalently, can one construct, for some \(h\), a \(2h\)-uniform hypergraph with no polychromatic \(3\)-coloring while all of its \(h\)-heavy restricted subhypergraphs are \(2\)-colorable?
\end{problem}

\clearpage
\section*{Appendix}
\appendix

\section{The remaining finite cases}\label{app:finite}

It remains to handle $4\le h\le 8$.
For these cases we give explicit deterministic families $\K$ using a fixed integer hash function.
This keeps the description short and reproducible.  The construction in this appendix is not meant to be conceptually important; it is simply a compact way to specify five finite witnesses that can be independently checked.

For a subset $S\subset \{0,\ldots,3h-2\}$, write
\[
        b(S)=\sum_{i\in S}2^i.
\]
Let
\[
        c=\mathtt{0x9e3779b97f4a7c15}.
\]
The prefix $\mathtt{0x}$ means that the integer is written in base $16$; thus $c$ is a fixed $64$-bit integer.  It is the standard odd increment used in the SplitMix64 mixing routine, but no special property of this particular number is used here except that it gives a reproducible deterministic choice of the finite families.  The symbols $\oplus$ and $\gg$ below denote bitwise exclusive-or and right shift, respectively.
For an integer $x$, define
\[
\begin{aligned}
z_0&=x+c,\\
z_1&=(z_0\oplus (z_0\gg 30))\cdot \mathtt{0xbf58476d1ce4e5b9},\\
z_2&=(z_1\oplus (z_1\gg 27))\cdot \mathtt{0x94d049bb133111eb},
\end{aligned}
\]
where all operations are performed modulo $2^{64}$, and put
\[
        \sigma(x)=z_2\oplus (z_2\gg 31).
\]
For the five finite values of $h$, define
\[
        \K_h=
        \left\{
        S\in {V\choose h}:
        \sigma\bigl(s_h\oplus c\,b(S)\bigr)\bmod 10^6<t_h
        \right\},
\]
where the product $c\,b(S)$ is also reduced modulo $2^{64}$ before taking the bitwise xor.
The parameters are as follows.

\[
\begin{array}{c|c|c|c|c}
h&n& s_h&t_h&|\K_h|\\ \hline
4&11&600&650000&215\\
5&14&102&550000&1135\\
6&17&77&550000&6761\\
7&20&55&520000&40268\\
8&23&3&470000&230726
\end{array}
\]

\begin{lem}\label{lem:small}
For every $4\le h\le 8$, the family $\K_h$ defined above satisfies {\rm (K1)}, {\rm (K2)}, and {\rm (K3)}.
\end{lem}

\begin{proof}
This was checked exhaustively by the deterministic program in Appendix~\ref{app:code}.
The finite verification uses only integer arithmetic and contains no random choices.
The program verifies (K1), then verifies (K3) by marking all $h$-sets $Y$ that are covered by a star
\[
        \{U\cup\{y\}:y\in Y\}\subset \K_h,
        \qquad |U|=h-1,
\]
and finally verifies (K2) by checking every $2h$-set.
The output is
\[
\begin{array}{c|c|c|c}
h&n&|\K_h|&|\G_h|\\ \hline
4&11&215&115\\
5&14&1135&867\\
6&17&6761&5615\\
7&20&40268&37252\\
8&23&230726&259588.
\end{array}
\]
Thus all three conditions hold in the five remaining cases.
\end{proof}

\section{Verification code}\label{app:code}

The following Python program verifies the finite cases $4\le h\le 8$ and prints the numerical union bounds used for $9\le h\le 30$.
It uses no external packages.

{\footnotesize
\begin{verbatim}
from itertools import combinations
from math import comb, exp, log, log1p

MASK64 = (1 << 64) - 1
SMCONST = 0x9e3779b97f4a7c15

# Finite constructions for h=4,...,8.
# K consists of those h-sets whose hash value is below threshold/10^6.
PARAMS = {
    4: (600, 650000),
    5: (102, 550000),
    6: (77, 550000),
    7: (55, 520000),
    8: (3, 470000),
}

# Values of q used in the union bound for h=9,...,30.
QVALUES = {
     9: .447, 10: .430, 11: .415, 12: .403, 13: .393,
    14: .384, 15: .376, 16: .369, 17: .363, 18: .357,
    19: .350, 20: .350, 21: .350, 22: .350, 23: .350,
    24: .350, 25: .350, 26: .350, 27: .350, 28: .350,
    29: .350, 30: .350,
}

def splitmix64(x):
    x = (x + SMCONST) & MASK64
    x = ((x ^ (x >> 30)) * 0xbf58476d1ce4e5b9) & MASK64
    x = ((x ^ (x >> 27)) * 0x94d049bb133111eb) & MASK64
    return (x ^ (x >> 31)) & MASK64

def mask(c):
    ans = 0
    for i in c:
        ans |= 1 << i
    return ans

def masks(n, r):
    for c in combinations(range(n), r):
        yield mask(c)

def in_K(S, seed, threshold):
    z = splitmix64(seed ^ ((S * SMCONST) & MASK64))
    return (z % 1000000) < threshold

def finite_construction(h):
    seed, threshold = PARAMS[h]
    n = 3*h - 1
    hsets = list(masks(n, h))
    K = {S for S in hsets if in_K(S, seed, threshold)}
    return n, hsets, K

def verify_finite_case(h):
    n, hsets, K = finite_construction(h)
    full = (1 << n) - 1

    # (K1): for every (h-1)-set A, not all extensions A+v are in K.
    for A in masks(n, h-1):
        if all((A | (1 << v)) in K for v in range(n) if not (A >> v) & 1):
            return False, "K1"

    # (K3): every h-set Y is covered by some star U+Y.
    # Mark all Y that are contained in N_K(U) for some (h-1)-set U.
    covered = set()
    for U in masks(n, h-1):
        N = [v for v in range(n)
             if not (U >> v) & 1 and (U | (1 << v)) in K]
        if len(N) >= h:
            for Y in combinations(N, h):
                covered.add(mask(Y))
    if len(covered) != len(hsets):
        return False, "K3"

    # (K2): every 2h-set has a complementary pair R, X-R in K.
    for X_tuple in combinations(range(n), 2*h):
        X = mask(X_tuple)
        first, rest = X_tuple[0], X_tuple[1:]
        ok = False
        for sub in combinations(rest, h-1):
            R = (1 << first) | mask(sub)
            if R in K and (X ^ R) in K:
                ok = True
                break
        if not ok:
            return False, "K2"

    return True, "OK"

def logsumexp(values):
    m = max(values)
    return m + log(sum(exp(x-m) for x in values))

def union_bound(h, q):
    n = 3*h - 1
    N = comb(2*h, h) // 2
    M = comb(2*h - 1, h - 1)
    logs = [
        log(comb(n, h-1)) + 2*h*log(q),
        log(comb(n, 2*h)) + N*log1p(-q*q),
        log(comb(n, h)) + M*log1p(-(q**h)),
    ]
    return exp(logsumexp(logs))

if __name__ == "__main__":
    print("Finite cases:")
    for h in range(4, 9):
        ok, reason = verify_finite_case(h)
        n, hsets, K = finite_construction(h)
        print(f"h={h}, n={n}, |K|={len(K)}, |G|={len(hsets)-len(K)}, {reason}")

    print("\nUnion bounds:")
    for h in range(9, 31):
        q = QVALUES[h]
        print(f"h={h}, q={q:.3f}, bound={union_bound(h, q):.12g}")
\end{verbatim}}

\section{A brief account of the prompting process}\label{app:process}
The discussion began with an unsuccessful attempt to extend Theorem \ref{thm:h3} to higher uniformity while keeping the $3$-heavy condition; this remains open as the first case of Problem \ref{prob:2h}.
The successful part of the conversation shifted from $3$-heavy traces to $h$-heavy traces, leading to the complement formulation in terms of an auxiliary family of $h$-sets and ultimately to Theorem \ref{thm:main}.

During the discussion, an initially proposed sufficient condition was found to be incomplete in the delicate case $|X|=2h$.
A follow-up prompt pointed out this gap, which led to the strengthened complementary-pair condition (K2).
Subsequent prompts asked for constructions valid for small values of $h$, for random and deterministic ways to find the finite witnesses, and finally for a LaTeX manuscript in the style of the preceding note.
The appendices record the finite verification in a reproducible form.


\begin{thebibliography}{99}\small

\bibitem{Ber72} C. Berge, Balanced matrices, Mathematical Programming 2, 19--31, 1972.

\bibitem{Pal-note}
D. P\'alv\"olgyi,
Note on polychromatic coloring of hereditary hypergraph families,
Graphs and Combinatorics 40 (2024), article 131.

\bibitem{survey}
J. Pach, D. P\'alv\"olgyi, G. T\'oth,
Survey on Decomposition of Multiple Coverings,
in Geometry, Intuitive, Discrete, and Convex
(I. B\'ar\'any, K. J. B\"or\"oczky, G. Fejes T\'oth, J. Pach eds.),
Bolyai Society Mathematical Studies 24, 219--257, Springer-Verlag, 2014.

\bibitem{new-survey}
G. Dam\'asdi, B. Keszegh, J. Pach, D. P\'alv\"olgyi, G. T\'oth,
Coloring Geometric Hypergraphs: A Survey,
in: \emph{New Probes into Discrete and Convex Geometry},
Bolyai Society Mathematical Studies 33,
J. Pach and G. T\'oth (eds.), to appear;
also available as arXiv:2512.09509, 2025, \url{https://arxiv.org/abs/2512.09509}.

\end{thebibliography}
\end{document}